\documentclass[preprint,12pt]{elsarticle}
\usepackage{amsthm,amsmath,amssymb}
\usepackage[colorlinks=true,citecolor=black,linkcolor=black,urlcolor=blue]{hyperref}
\usepackage{graphicx}
\usepackage{float}
\usepackage{mathtools}
\usepackage{epstopdf}
\usepackage{epsfig}
\usepackage{subfigure}
\usepackage{multirow}

\theoremstyle{plain}
\newtheorem{theorem}{Theorem}
\newtheorem{lemma}[theorem]{Lemma}
\newtheorem{corollary}[theorem]{Corollary}
\newtheorem{proposition}[theorem]{Proposition}

\theoremstyle{definition}
\newtheorem{definition}[theorem]{Definition}

\newtheorem{problem}[theorem]{Problem}

\theoremstyle{remark}
\newtheorem{remark}[theorem]{Remark}

\title{Twist polynomial interpolation for binary delta-matroids}

\author{Zhao Zhao,
Qi Yan\footnote{Corresponding author.}\\
\small School of Mathematics and Statistics, Lanzhou University, PR China\\
\small{\tt Email: zhzhao2025@lzu.edu.cn; yanq@lzu.edu.cn}}
\date{}
\journal{arXiv}

\begin{document}
\begin{abstract}
Gross, Mansour and Tucker introduced the partial-dual polynomial of a ribbon graph and asked under what conditions such a polynomial is even-interpolating, odd-interpolating, or both.
In this paper, we provide an answer to this open problem.
Using the framework of delta-matroids, we prove that the twist polynomial of any binary delta-matroid is either an even polynomial, an odd polynomial, or both even-interpolating and odd-interpolating.
Applying this to ribbon graphs, we deduce that the partial-dual polynomial of any ribbon graph satisfies the same conclusion.
\end{abstract}

\maketitle

\section{Introduction}

Chmutov \cite{CSL2007} introduced the notion of partial duality for ribbon graphs. This construction was used to study signed Bollob\'as-Riordan polynomials \cite{BR2, BR} and knot polynomials \cite{CP2007, K1987}.  Gross, Mansour and Tucker \cite{GMT2020} introduced the partial-dual polynomial for an arbitrary ribbon graph.
\begin{definition}[\cite{GMT2020}]
    The partial-dual polynomial of any ribbon graph $G$ is the generating function
    \[
{^\partial\varepsilon_G(z)}=\sum_{A\subseteq E(G)}z^{\varepsilon(G^A)}
    \]
    that enumerates all partial duals of $G$ by Euler genus.
\end{definition}

Gross, Mansour and Tucker \cite{GMT2020} studied various basic properties of partial-dual  polynomials, including interpolation and log-concavity. They showed that for any orientable ribbon graph, the polynomial is even-interpolating, and posed the following open problem for the non-orientable case.

\begin{problem}[\cite{GMT2020}]\label{problem}
Under what interesting sets of sufficient conditions on a non-orientable ribbon graph is the partial-dual polynomial  even-interpolating, odd-interpolating, or both?
\end{problem}

Recently, partial-dual polynomials have been further investigated, for example in \cite{Chen2026, Cheng, CG2, DJY, QYXJ}. In this paper, we use the tool of delta-matroids to solve Problem~\ref{problem}. Delta-matroids, introduced by Bouchet \cite{GA1987}, generalize embedded graphs in the same way that matroids generalize graphs; see \cite{CH19, CH2019} for further background.  Yan and Jin \cite{YQ2022} extended the partial-dual polynomial to the twist polynomial of a delta-matroid.

\begin{definition}[\cite{YQ2022}]
    The twist polynomial of any set system \( D = (E, \mathcal{F}) \) is the generating function
    \[
    {^\partial w_D(z)}=\sum_{A\subseteq E}z^{w(D*A)}
    \]
    that enumerates all twists of $D$ by width.
\end{definition}

We first prove the following result for binary delta-matroids, which  determines the interpolating behaviour of their twist polynomials.

\begin{theorem}\label{theorem5}
For a binary delta-matroid $D=(E,\mathcal{F})$, the twist polynomial ${^\partial w_D(z)}$ is either an even polynomial, an odd polynomial, or both even-interpolating and odd-interpolating.
\end{theorem}

Using the relation between ribbon graphs and delta-matroids, Theorem~\ref{theorem5} specializes to ribbon graphs and thereby answers Problem~\ref{problem}.

\begin{theorem}\label{main theorem}
For any ribbon graph $G$, the partial-dual polynomial ${^\partial\varepsilon_G(z)}$ is either an even polynomial, an odd polynomial, or both even-interpolating and odd-interpolating.
\end{theorem}

\section{Preliminaries}

In this section, we recall several definitions related to delta-matroids and ribbon graphs, and refer the reader to \cite{CH19, CH2019, EM1, EM} for further details.

\subsection{Delta-matroids}

A \emph{set system} is a pair $D = (E, \mathcal{F})$, where $E$, also denoted by $E(D)$, is a finite set, called the \emph{ground set},
and $\mathcal{F}$, also denoted by $\mathcal{F}(D)$, is a collection of subsets of $E$, called \emph{feasible sets}.
$D$ is \emph{proper} if $\mathcal{F} \neq \emptyset$, and is \emph{normal} if $\emptyset\in\mathcal{F}$.
For sets $X$ and $Y$, their \emph{symmetric difference} is $$X \Delta Y = (X \cup Y) - (X \cap Y).$$
$D$ is \emph{even} if $|X \Delta Y|$ is even for all $X, Y \in \mathcal{F}(D)$; otherwise it is \emph{odd}.
Throughout the paper, we often omit set brackets for singletons: for example, we write $E - e$ instead of $E - \{e\}$, and $F \cup e$ instead of $F \cup \{e\}$.

Bouchet \cite{GA1987} introduced delta-matroids and the notion of a twist as follows.

\begin{definition}[\cite{GA1987}]
A \emph{delta-matroid} is a proper set system $D=(E, \mathcal{F})$ satisfying the \emph{Symmetric Exchange Axiom}:
for any $X, Y \in \mathcal{F}$ and any $u \in X \Delta Y$, there exists $v \in X \Delta Y$ (possibly $v = u$) such that
$X \Delta \{u, v\} \in \mathcal{F}$.
\end{definition}

The \emph{twist} of a set system $D=(E,\mathcal{F})$ with respect to a subset $A \subseteq E$, denoted by $D * A$,
is the set system $$(E, \{A \Delta X \mid X \in \mathcal{F}\}).$$
The \emph{dual} of $D$ is $D^{*} = D * E$.
Note that the twist of a
delta-matroid is also a delta-matroid \cite{GA1987}.

For a set system $D = (E, \mathcal{F})$, let $\mathcal{F}_{\max}(D)$ and $\mathcal{F}_{\min}(D)$ be the collections of
maximum and minimum cardinality feasible sets of $D$, respectively. Let $D_{\max} = (E, \mathcal{F}_{\max}(D))$ and $D_{\min} = (E, \mathcal{F}_{\min}(D))$.
Let $r(D_{\max})$ and $r(D_{\min})$ denote the sizes of the largest and smallest feasible sets of $D$, respectively.
The \emph{width} of $D$, denoted by $w(D)$, is defined by
\[
w(D) = r(D_{\max}) - r(D_{\min}).
\]
For all non-negative integers $i \leq w(D)$, let
\[
\mathcal{F}_{\max - i}(D) = \{F \in \mathcal{F} \mid |F| = r(D_{\max}) - i\}
\]
and
\[
\mathcal{F}_{\min + i}(D) = \{F \in \mathcal{F} \mid |F| = r(D_{\min}) + i\}.
\]
The \emph{maximum twist width} $w_M(D)$ of a set system $D=(E,\mathcal{F})$ is defined by
\[
w_M(D) = \max \{ w(D * A) \mid A \subseteq E \}.
\]
We say that an integer $t$ is a \emph{twist width} of $D$ if there exists $A\subseteq E$ such that $w(D*A)=t$.

Let $D=(E, \mathcal{F})$ be a set system. An element $e \in E$ is a \emph{coloop} if $e \in F$ for every $F \in \mathcal{F}$,
and it is a \emph{loop} if $e \notin F$ for every $F \in \mathcal{F}$.
Let $e \in E$. The \emph{deletion} of $e$ from $D$, denoted $D - e$, is the set system $(E - e, \mathcal{F}')$ where
\[
\mathcal{F}' =
\begin{cases}
\{F \in \mathcal{F} \mid F \subseteq E - e\}, & \text{if $e$ is not a coloop}, \\[4pt]
\{F - e \mid F \in \mathcal{F}\}, & \text{if $e$ is a coloop}.
\end{cases}
\]
Bouchet \cite{BD1991} showed that the order of deletions is immaterial.
For a subset $A \subseteq E$, we define $D-A$ as the result of deleting every element of $A$ in any order.
If $D$ is a delta-matroid, then $D-A$ is also a delta-matroid (see \cite{BD1991}).

For a finite set $E$, let $C$ be a symmetric $|E| \times |E|$ matrix over $\mathrm{GF}(2)$ with rows and columns indexed by $E$.
For $A \subseteq E$, denote by $C[A]$ the principal submatrix of $C$ induced by $A$.
Define the set system $D(C) = (E, \mathcal{F})$ by
\[
\mathcal{F} = \{ A \subseteq E \mid C[A] \text{ is non-singular} \},
\]
where, by convention, $C[\emptyset]$ is non-singular. Bouchet \cite{AB4} showed that $D(C)$ is a normal delta-matroid.

A delta-matroid is \emph{binary} if it has a twist that is isomorphic to $D(C)$ for some symmetric matrix $C$ over $\mathrm{GF}(2)$. In particular, if $D=(E, \mathcal{F})$ is a normal binary delta-matroid, then there exists a unique symmetric $|E|\times|E|$ matrix $C$ over $\mathrm{GF}(2)$, whose rows and columns are labelled (in the same order) by the set $E$ such that $D=D(C)$. In fact, the matrix $C$
can be constructed  as follows \cite{BD1991, MO2019}:
\begin{description}
\item [(1)] Set $C_{vv}=1$ if and only if $\{v\}\in \mathcal{F}$. This determines the diagonal entries of $C$;
\item [(2)] Set $C_{uv}=1$ if and only if $\{u\}, \{v\}\in \mathcal{F}$ but $\{u, v\}\notin \mathcal{F}$, or $\{u, v\}\in \mathcal{F}$ but $\{u\}$ and $\{v\}$ are not both in $\mathcal{F}$. Then the feasible sets of size at most two determine the off-diagonal entries of $C$.
\end{description}

\begin{definition}[\cite{CH2019}]
Let $D=(E,\mathcal{F})$ be a set system and $e\in E$.
\begin{description}
    \item [(1)] $e$ is a \emph{ribbon loop} if $e$ is a loop in $D_{\min}$;
    \item [(2)] A ribbon loop $e$ is \emph{non-orientable} if it remains a ribbon loop in $D * e$, and \emph{orientable} otherwise.
\end{description}
\end{definition}

Let $D = (E, \mathcal{F})$ be a set system and $e \in E$. The \emph{primal type} of $e$ in $D$ is $p$, $u$, or $t$ according to whether $e$ is a non-ribbon loop, an orientable loop, or a non-orientable loop, respectively. The primal type of $e$ in the dual $D^{*}$ is called its \emph{dual type} in $D$. Together, the primal and dual types form the \emph{type} of $e$, written as a juxtaposed pair of letters (primal type first). For example, type $pu$ means primal type $p$ and dual type $u$.

\begin{proposition}[\cite{YQ2024}]\label{lemma6}
For a set system $D=(E,\mathcal{F})$ and $e\in E$, the following statements hold.
\begin{description}
\item[\textup{(1)}] The primal type of $e$ is $p$ in $D$ if and only if there exists $F \in \mathcal{F}_{\mathrm{min}}(D)$ such that $e \in F$.
\item[\textup{(2)}] The primal type of $e$ is $u$ in $D$ if and only if for every $F \in \mathcal{F}_{\mathrm{min}}(D) \cup \mathcal{F}_{\mathrm{min}+1}(D)$, $e \notin F$.
\item[\textup{(3)}] The primal type of $e$ is $t$ in $D$ if and only if for every $F \in \mathcal{F}_{\mathrm{min}}(D)$, $e \notin F$, and there exists $F_1 \in \mathcal{F}_{\mathrm{min}+1}(D)$ such that $e \in F_1$.

\item[\textup{(4)}] The dual type of $e$ is $p$ in $D$ if and only if there exists $F \in \mathcal{F}_{\mathrm{max}}(D)$ such that $e \notin F$.
\item[\textup{(5)}] The dual type of $e$ is $u$ in $D$ if and only if for every $F \in \mathcal{F}_{\mathrm{max}}(D) \cup \mathcal{F}_{\mathrm{max}-1}(D)$, $e \in F$.
\item[\textup{(6)}] The dual type of $e$ is $t$ in $D$ if and only if for every $F \in \mathcal{F}_{\mathrm{max}}(D)$, $e \in F$, and there exists $F_1 \in \mathcal{F}_{\mathrm{max}-1}(D)$ such that $e \notin F_1$.
\end{description}
\end{proposition}

For a set system $D = (E, \mathcal{F})$, the width of $D*e$ for each type of $e \in E$ is given in \cite{YQ2024} and reproduced in Table~\ref{tab:widths}. Note that in all cases $$|w(D*e)-w(D)| \le 2.$$ Although the table is stated in \cite{YQ2024} for vf-safe delta-matroids, the verification for each type uses only the definition of twist and applies equally to arbitrary set systems.

\begin{table}[t]
  \centering
  \caption{The width of $D*e$}
  \label{tab:widths}
  \begin{tabular}{llll}
    \hline\noalign{\smallskip}
    Type of $e$ & $r(D*e)_{\text{min}}$ & $r(D*e)_{\text{max}}$ & $w(D*e)$\\
    \noalign{\smallskip}\hline\noalign{\smallskip}
    $pp$ & $r(D_{\text{min}})-1$ & $r(D_{\text{max}})+1$ & $w(D)+2$ \\
    $uu$ & $r(D_{\text{min}})+1$ & $r(D_{\text{max}})-1$ &  $w(D)-2$ \\
    $pu$ & $r(D_{\text{min}})-1$ & $r(D_{\text{max}})-1$ & $w(D)$ \\
    $up$ & $r(D_{\text{min}})+1$ & $r(D_{\text{max}})+1$ &   $w(D)$  \\
    $tp$ & $r(D_{\text{min}})$ & $r(D_{\text{max}})+1$ &   $w(D)+1$ \\
    $tu$ & $r(D_{\text{min}})$ & $r(D_{\text{max}})-1$ &  $w(D)-1$ \\
    $pt$ & $r(D_{\text{min}})-1$ & $r(D_{\text{max}})$ &  $w(D)+1$ \\
    $ut$ & $r(D_{\text{min}})+1$ & $r(D_{\text{max}})$ &  $w(D)-1$ \\
    $tt$ & $r(D_{\text{min}})$ & $r(D_{\text{max}})$ &   $w(D)$ \\
    \noalign{\smallskip}\hline
  \end{tabular}
\end{table}

\subsection{Ribbon graphs}

\begin{definition}[\cite{BR}]
A {\it ribbon graph} $G$ is a (orientable or non-orientable) surface with boundary,
represented as the union of two sets of topological discs, a set $V(G)$ of vertices, and a set $E(G)$ of edges, subject to the following conditions.
\begin{description}
\item[(1)] The vertices and edges intersect in disjoint line segments.
\item[(2)] Each such line segment lies on the boundary of precisely one vertex and precisely one edge.
\item[(3)] Every edge contains exactly two such line segments.
\end{description}
\end{definition}

A ribbon graph is \emph{orientable} if its underlying surface is orientable; otherwise, it is \emph{non-orientable}.
 If $G$ is a ribbon graph, we denote by $f(G)$ the number of boundary components of $G$, and we define $v(G)$, $e(G)$, and $c(G)$ to be the number of vertices, edges, and connected components of $G$, respectively. We let \[\chi(G)=v(G)-e(G)+f(G),\] the usual \emph{Euler characteristic}, which holds for any ribbon graph, connected or not. The \emph{Euler genus} of $G$ is
\[
\varepsilon(G)=2c(G)-\chi(G).
\]

\begin{definition}[\cite{CSL2007}]
For a ribbon graph $G$ and $A\subseteq E(G)$,  the \emph{partial dual} $G^{A}$ of $G$ with respect to $A$ is a ribbon graph obtained from $G$ by gluing a disc to $G$ along each boundary component of the spanning ribbon subgraph $(V(G), A)$ $($such discs will be the vertex-discs of $G^{A})$, removing the interiors of all the vertex-discs of $G$ and keeping the edge-discs unchanged.
\end{definition}

The \emph{maximum partial-dual Euler-genus $\varepsilon_M(G)$} of a ribbon graph $G$ is
\[
\varepsilon_M(G) = \max\{\, \varepsilon(G^A) \mid A \subseteq E(G) \,\}.
\]

A \emph{quasi-tree} is a ribbon graph with exactly one boundary component. A ribbon subgraph $H$ of a connected ribbon graph $G$ is a \emph{spanning quasi-tree} if $H$ is a quasi-tree and has the same vertex set as $G$. If $G$ is not connected then we say a ribbon subgraph $H$ is a \emph{spanning quasi-tree} of $G$ if it is a disjoint union of spanning quasi-trees of all connected components of $G$.

\begin{definition}[\cite{CH2019}]
Let $G$ be a ribbon graph and let
\[
\mathcal{F} = \{ F \subseteq E(G) \mid \text{$F$ is the edge set of a spanning quasi-tree of $G$} \}.
\]
The delta-matroid of $G$ is $D(G) = (E(G), \mathcal{F})$. We say that a delta-matroid is \emph{ribbon-graphic} if it is isomorphic to $D(G)$ for some ribbon graph $G$.
\end{definition}

For a polynomial $p(z)=\sum_{i=0}^{n}c_i z^i$, we say that $p(z)$ has a \emph{gap of size $k$ at coefficient $c_i$} if $c_{i-1}c_{i+k}\neq 0$ but $c_i=c_{i+1}=\cdots=c_{i+k-1}=0$.
The polynomial $p(z)$ is \emph{interpolating} if it is non-zero and has no gaps.
Write $p(z)=p_e(z^2)+z\,p_o(z^2)$, where $p_e(z^2)$ and $p_o(z^2)$ consist of the even-degree and odd-degree terms of $p(z)$, respectively.
We call $p(z)$ \emph{even-interpolating} (resp.\ \emph{odd-interpolating}) if $p_e(z^2)$ (resp.\ $p_o(z^2)$) is interpolating.
An \emph{even} (resp.\ \emph{odd}) polynomial is a polynomial such that the only terms that have non-zero coefficients are the terms of even (resp.\ odd) degree.

\section{Proof of main results}

\begin{lemma}[\cite{BC2021}] \label{lemma1}
If $X$ is any feasible set in a delta-matroid $D$, then there exist
$A\in \mathcal{F}_{\mathrm{min}}(D)$ and $B\in \mathcal{F}_{\mathrm{max}}(D)$
such that $A\subseteq X \subseteq B$.
\end{lemma}

\begin{lemma}[\cite{BH2013}]
\label{corollary2}
For a delta-matroid $D=(E,\mathcal{F})$ and $e\in E$, if the primal type of $e$ is $t$ in $D$, then $$\mathcal{F}_{\mathrm{min}}(D*e) = \mathcal{F}_{\mathrm{min}}(D).$$

\end{lemma}

\begin{lemma}\label{lem:prim_delete}
Let $D=(E,\mathcal{F})$ be a delta-matroid and $e_1,e_2\in E$ with $e_1\neq e_2$.
Then the primal type of $e_2$ in $D- e_1$ is the same as its primal type in $D$.
\end{lemma}

\begin{proof}
Let $r = r(D_{\min})$. We consider two cases.

\noindent \textbf{Case 1:} $e_1$ is a coloop.
Then every feasible set contains $e_1$, so
\[
D- e_1 = (E-e_1,\;\{F-e_1\mid F\in\mathcal{F}\}).
\]
Consequently,
\[
\begin{aligned}
\mathcal{F}_{\min}(D- e_1) &= \{F - e_1 \mid F \in \mathcal{F}_{\min}(D)\},\\
\mathcal{F}_{\min+1}(D- e_1) &= \{F - e_1 \mid F \in \mathcal{F}_{\min+1}(D)\}.
\end{aligned}
\]
We verify that the primal type of $e_2$ is preserved in each of the three subcases.

\begin{description}
\item[($p$)] Since the primal type of $e_2$ in $D$ is $p$, by Proposition~\ref{lemma6}(1), there exists $F\in\mathcal{F}_{\min}(D)$ with $e_2\in F$.
Then $$e_2\in F-e_1\in\mathcal{F}_{\min}(D- e_1).$$
By Proposition~\ref{lemma6}(1), the primal type of $e_2$ in $D- e_1$ is $p$.

\item[($u$)] For any $F'\in\mathcal{F}_{\min}(D- e_1)\cup\mathcal{F}_{\min+1}(D- e_1)$,
there exists $F\in\mathcal{F}_{\min}(D)\cup\mathcal{F}_{\min+1}(D)$ such that $F' = F - e_1$.
Since the primal type of $e_2$ in $D$ is $u$, by Proposition~\ref{lemma6}(2), $e_2\notin F$.
Hence $e_2\notin F'$. By Proposition~\ref{lemma6}(2), the primal type of $e_2$ in $D- e_1$ is $u$.

\item[($t$)] For any $F'\in\mathcal{F}_{\min}(D- e_1)$, there exists $F\in\mathcal{F}_{\min}(D)$ such that $F' = F - e_1$.
Since the primal type of $e_2$ in $D$ is $t$, Proposition~\ref{lemma6}(3) gives $e_2\notin F$; hence $e_2\notin F'$.
Moreover, there exists $F_1\in\mathcal{F}_{\min+1}(D)$ with $e_2\in F_1$.
Then $F_1 - e_1 \in \mathcal{F}_{\min+1}(D- e_1)$ and $e_2\in F_1 - e_1$.
By Proposition~\ref{lemma6}(3), the primal type of $e_2$ in $D- e_1$ is $t$.
\end{description}

\noindent \textbf{Case 2:} $e_1$ is not a coloop.
Then there exists $F_0 \in \mathcal{F}$ with $e_1 \notin F_0$.
By Lemma~\ref{lemma1}, there exists $M \in \mathcal{F}_{\min}(D)$ such that $M \subseteq F_0$. Hence $e_1 \notin M$.
Thus $M \in \mathcal{F}_{\min}(D - e_1)$, and $r_{\min}(D - e_1) = r$.
From the definition of deletion it follows that
\[
\begin{aligned}
\mathcal{F}_{\min}(D - e_1) &= \{F \in \mathcal{F}_{\min}(D) \mid e_1 \notin F\},\\
\mathcal{F}_{\min+1}(D - e_1) &= \{F \in \mathcal{F}_{\min+1}(D) \mid e_1 \notin F\}.
\end{aligned}
\]
We consider the three possibilities for the primal type of $e_2$ in $D$.

\begin{description}
\item[($u$)] For any $F'\in\mathcal{F}_{\min}(D- e_1)\cup\mathcal{F}_{\min+1}(D- e_1)$, we have $$F'\in\mathcal{F}_{\min}(D)\cup\mathcal{F}_{\min+1}(D).$$
Since the primal type of $e_2$ in $D$ is $u$, by Proposition~\ref{lemma6}(2), $e_2\notin F'$. Hence, the primal type of $e_2$ in $D- e_1$ is $u$.

\item[($t$)] For any $F'\in\mathcal{F}_{\min}(D - e_1)$, we have $F'\in\mathcal{F}_{\min}(D)$.
Since the primal type of $e_2$ in $D$ is $t$, Proposition~\ref{lemma6}(3) gives $e_2\notin F'$.
To verify the second requirement of Proposition~\ref{lemma6}(3), choose $F_2\in\mathcal{F}_{\min+1}(D)$ with $e_2\in F_2$.
If $e_1\notin F_2$, then $F_2\in\mathcal{F}_{\min+1}(D - e_1)$.
If $e_1\in F_2$, we construct another suitable set as follows.
By Lemma~\ref{corollary2}, $\mathcal{F}_{\min}(D*e_2)=\mathcal{F}_{\min}(D)$.
Choose $M\in\mathcal{F}_{\min}(D)$ with $e_1\notin M$ (such an $M$ exists by Lemma~\ref{lemma1}).
Since the primal type of $e_2$ in $D$ is $t$, Proposition~\ref{lemma6}(3) implies $e_2\notin M$.
Because $$M\in\mathcal{F}_{\min}(D)=\mathcal{F}_{\min}(D*e_2),$$ we have $M\Delta e_2\in\mathcal{F}(D)$.
As $e_2\notin M$, $M\Delta e_2 = M\cup e_2$, hence $M\cup e_2\in\mathcal{F}(D)$ and $|M\cup e_2| = r+1$.
Thus $M\cup e_2\in\mathcal{F}_{\min+1}(D)$.
Since $e_1\notin M$ and $e_2\neq e_1$, we have $e_1\notin M\cup e_2$. Hence $M\cup e_2\in\mathcal{F}_{\min+1}(D - e_1)$ and $e_2\in M\cup e_2$.
Thus in either subcase there exists a set in $\mathcal{F}_{\min+1}(D - e_1)$ containing $e_2$.
By Proposition~\ref{lemma6}(3), the primal type of $e_2$ in $D - e_1$ is $t$.

\item[($p$)] We first show that there exists $F\in \mathcal{F}_{\min}(D)$ such that $e_1\notin F$ and $e_2\in F$.
Assume the contrary: every $F \in \mathcal{F}_{\min}(D)$ that contains $e_2$ also contains $e_1$.
Since the primal type of $e_2$ in $D$ is $p$, there exists $F_1\in \mathcal{F}_{\min}(D)$ with $e_2 \in F_1$ by Proposition~\ref{lemma6}(1).
The assumption then forces $e_1 \in F_1$.
Because $e_1$ is not a coloop, there exists $F_2 \in \mathcal{F}_{\min}(D)$ with $e_1 \notin F_2$.
Apply the Symmetric Exchange Axiom to $F_1, F_2$ and $u = e_1 \in F_1 \setminus F_2$.
There exists $v \in F_1 \Delta F_2$ such that $F_1 \Delta \{e_1, v\} \in \mathcal{F}(D)$.
We determine the possible size of $F_1 \Delta \{e_1, v\}$.
Since $|F_1| = r$, the size depends on $v$ as follows:
\begin{itemize}
\item If $v = e_1$, then $|F_1 \Delta \{e_1\}| = r - 1$, which is impossible because $r$ is the minimum size of a feasible set.
\item If $v \in F_1 \setminus \{e_1\}$, then $F_1 \Delta \{e_1, v\} = F_1 \setminus \{e_1, v\}$ has size $r - 2$, again a contradiction.
\item Otherwise, $v \in F_2 \setminus F_1$. In this case $F_1 \Delta \{e_1, v\} = (F_1 \setminus \{e_1\}) \cup \{v\}$ has size $r$, so $F_1 \Delta \{e_1, v\}\in\mathcal{F}_{\min}(D)$.
Because $e_2 \in F_1$ and $e_2 \neq e_1$, we have $e_2 \in F_1 \setminus \{e_1\}$ and  $v \neq e_2$.
Hence $e_2\in F_1 \Delta \{e_1, v\}$, but $e_1\notin F_1 \Delta \{e_1, v\}$, contradicting the assumption that every minimum feasible set containing $e_2$ also contains $e_1$.
\end{itemize}
Hence there exists $F_0 \in \mathcal{F}_{\min}(D)$ with $e_2 \in F_0$ and $e_1 \notin F_0$.
Then $F_0 \in \mathcal{F}_{\min}(D - e_1)$, and Proposition~\ref{lemma6}(1) yields that the primal type of $e_2$ in $D - e_1$ is $p$.
\end{description}
\end{proof}

\begin{remark}
Lemma~\ref{lem:prim_delete} is not true for set systems. For example, let
\[
D = (\{1,2,3,4,5,6\},\, \{\{1,2\},\{3,4,5\},\{3,4,5,6\}\}).
\]
Then
\[
D-\{1\} = (\{2,3,4,5,6\},\, \{\{3,4,5\},\{3,4,5,6\}\}).
\]
We have
\begin{itemize}
    \item $2$ has primal type $p$ in $D$, but $u$ in $D-\{1\}$;
    \item $3$ has primal type $t$ in $D$, but $p$ in $D-\{1\}$;
    \item $6$ has primal type $u$ in $D$, but $t$ in $D-\{1\}$.
\end{itemize}
\end{remark}

\begin{lemma}
    Let $D = (E, \mathcal{F})$ be a binary delta-matroid. If every element of $E$ has primal type $u$ in $D$, then $D$ is even.
    \label{Lemma4}
\end{lemma}

\begin{proof}
Since every $e \in E$ has primal type $u$ in $D$, it follows that $e \notin F$ for all
$F \in \mathcal{F}_{\min}(D) \cup \mathcal{F}_{\min+1}(D)$ by Proposition~\ref{lemma6}(2).
Consequently $\mathcal{F}_{\min}(D) = \{\emptyset\}$ and $\mathcal{F}_{\min+1}(D) = \emptyset$, i.e., $D$ is normal and $r_{\min}(D) = 0$, and there are no feasible singletons. Then there exists a unique symmetric
matrix $C$ over $\mathrm{GF}(2)$ with rows and columns indexed by $E$ such that $D = D(C)$.
The diagonal entry $C_{vv}$ equals $1$ if and only if $\{v\} \in \mathcal{F}$.
As no singleton is feasible, we have $C_{vv} = 0$ for every $v \in E$.

Now suppose that $D$ is not even. Then there exists a feasible set $A \in \mathcal{F}$
with $|A|$ odd. The principal submatrix $C[A]$ is a symmetric matrix over $\mathrm{GF}(2)$ with all diagonal
entries equal to $0$. It is a basic fact that the rank of such a matrix is even.
Since $|A|$ is odd, $$\operatorname{rank}(C[A]) \le |A|-1 < |A|.$$ Hence $\det(C[A]) = 0$.
This contradicts $A \in \mathcal{F}$, because feasible sets correspond to non-singular principal submatrices.
Therefore every feasible set has even cardinality, and $D$ is even.
\end{proof}

\begin{remark}
Lemma~\ref{Lemma4} does not hold for non-binary delta-matroids. For example, let
$$D = (\{1,2,3\},\, \{\emptyset, \{1,2\}, \{1,3\}, \{2,3\}, \{1,2,3\}\}).$$
This is a non-binary delta-matroid (see \cite{BD1991}).
While every element of $E$ has primal type $u$ in $D$, the delta-matroid $D$ is odd.
\end{remark}

\begin{lemma}[\cite{CH2019}]
\label{lemma8}
Let $D=(E,\mathcal{F})$ be a delta-matroid and let $A,B\subseteq E$ with $A\cap B=\emptyset$. Then
\[
(D * A) - B = (D - B) * A.
\]
\end{lemma}

\begin{theorem} \label{theorem3}
Let $D=(E,\mathcal{F})$ be a binary delta-matroid. If the twist width set $\{\,w(D*A) \mid A \subseteq E\,\}$ contains two consecutive integers $k$ and $k+1$, then for every integer $m$ with $k+2 \le m \le w_M(D)$, $m\in \{\,w(D*A) \mid A \subseteq E\,\}$.
\end{theorem}

\begin{proof}
Since $k\in \{\,w(D*A) \mid A \subseteq E\,\}$, there exists $A \subseteq E$ such that $w(D * A) = k$. If $w_M(D) \le k + 1$, there is nothing to prove.
Thus we assume $w_M(D) \ge k + 2$. To complete the proof, it suffices to show that there exists $A' \subseteq E$ such that $w(D * A') = k + 2$.
Note that once such an $A'$ is obtained, the pair $(k+1, k+2)$ consists of two consecutive twist widths, and by iterating the same argument we obtain every integer $m$ with $k+2 \le m \le w_M(D)$.

If there exists $F \subseteq E$ with $w(D * A * F) = k+2$, the proof is complete.
Suppose that $k+2$ is not a twist width of $D * A$.
Since each twist by a single element changes the width by at most $2$ (Table~\ref{tab:widths}) and $k+2$ is forbidden, the set of twist widths cannot skip two consecutive integers.
As $w_M(D) \ge k+2$, it follows that $k+3$ must be a twist width of $D * A$.
Thus there exists $Y \subseteq E$ with $w(D * A * Y) = k+3$.

Now choose $B = \{e_1, e_2, \dots, e_n\}$ to be a set of minimum cardinality among all sets with
$w(D * A * B) = k+3$.
For any $e_x \in B$, consider the delta-matroid $D * A * (B - e_x)$,
which is $D * A * B * e_x$.
By Table~\ref{tab:widths}, the effect on the width depends on the type of $e_x$ in $D * A * B$.

\begin{itemize}
\item If $e_x$ is of type $tt$, $pu$, or $up$ in $D * A * B$, then
$$w(D * A * (B - e_x)) = w(D * A * B) = k+3,$$
contradicting the minimality of $|B|$.

\item If $e_x$ is of type $pp$, $pt$, or $tp$ in $D * A * B$, then twisting $e_x$ increases the width.
More precisely, $$w(D * A * (B - e_x)) = w(D * A * B * e_x) = k+3 + \delta,$$
where $\delta = 2$ for $pp$, and $\delta = 1$ for $pt$ or $tp$.  Hence $k+3 + \delta \ge k+4$.
Now start from $D * A$ (width $k$) and twist the elements of $B - e_x$ in any order.
By Table~\ref{tab:widths} each step changes the width by at most $2$, and the width $k+2$ never occurs.
Consequently, before reaching a width of at least $k+4$, the sequence must at some point attain $k+3$;
otherwise, it would have to jump from a value at most $k+1$ directly to a value at least $k+4$, a change of at least $3$, which is impossible.
Thus some proper subset of $B$ yields width $k+3$, contradicting the minimality of $|B|$.
\end{itemize}
Hence every element $e_x \in B$ must be
of type $uu$, $ut$, or $tu$ in $D * A * B$. If some $e_x \in B$ is of type $ut$ or $tu$ in $D * A * B$, then by Table~\ref{tab:widths}
\[
w(D * A * B * e_x) = w(D * A * B) - 1 = k+2,
\]
contradicting the assumption that $k+2$ is not a twist width of $D * A$.
Therefore every element $e_x \in B$ must be of type $uu$ in $D * A * B$.

Let $D' = D * A * B$.
Delete all elements not in $B$ to obtain $D' - B^c$.
Since every $e \in B$ is of type $uu$ in $D'$, the primal type of $e$ in $D'$ is $u$.
By Lemma~\ref{lem:prim_delete}, deleting $B^c$ does not change the primal type of any element of $B$.
Hence in $D' - B^c$, every element of $B$ still has primal type $u$.
Because $D'$ is a twist of the binary delta-matroid $D$ and deletion preserves binaryness,
$D' - B^c$ is a binary delta-matroid. Lemma~\ref{Lemma4} then implies that $D' - B^c$ is even.

Now fix an order $e_1, e_2, \dots, e_n$ of the elements of $B$.
For $m \in \{1,\dots,n\}$ set $F_m = \{e_1,\dots,e_{m-1}\}$, with the convention $F_1 = \emptyset$.
We determine the type of $e_m$ in $D' * F_m$. By Lemma~\ref{lemma8}, $$(D' * F_m) - B^c = (D' - B^c) * F_m.$$
By Lemma~\ref{lem:prim_delete}, deleting $B^c$ does not alter the primal type of any element of $B$. Hence the primal type of $e_m$ in $D' * F_m$ coincides with its primal type in $(D' - B^c) * F_m$.
Since $D' - B^c$ is even and twisting preserves evenness,  it follows that $(D' - B^c) * F_m$ is even. In an even delta-matroid no element can have primal type $t$. Thus the primal type of $e_m$ in $(D' - B^c) * F_m$ is either $p$ or $u$. Consequently, the primal type of $e_m$ in $D' * F_m$ is also $p$ or $u$.

A completely symmetric argument applied to the dual ${D'}^* = D' * E$ shows that the dual type of $e_m$ in $D' * F_m$ is also $p$ or $u$. Consequently, in each intermediate delta-matroid $D' * F_m$, the element $e_m$ can only be of type $uu$, $up$, $pu$ or $pp$.

Now start from $D'$ (width $k+3$) and twist the elements of $B$ again in the same order
$e_1, e_2, \dots, e_n$. For $m \in \{1,\dots,n\}$, when we twist $e_m$, the current delta-matroid
is $D' * F_m$. Note that the type of $e_m$ in $D' * F_m$ is $uu$, $up$, $pu$, or $pp$.
Thus, by Table~\ref{tab:widths}, twisting $e_m$ changes the width by an even number:
$-2$ for $uu$, $0$ for $up$ and $pu$, and $+2$ for $pp$.
After $n$ such twists we obtain $D' * B = D * A$, whose width is $k$.
Therefore the total change in width from $D'$ to $D * A$ is $k - (k+3) = -3$, an odd number,
which is impossible because it is a sum of even numbers.
This contradiction shows that our assumption that no set yields width $k+2$ is false.
Hence there exists $A' \subseteq E$ with $w(D *A') = k+2$, completing the proof.
\end{proof}

\begin{proof}[{\bf Proof of Theorem~\ref{theorem5}}.]
By the definition of the twist polynomial, the set of degrees with non-zero coefficient is precisely the set of twist widths
\[
\mathcal{W}=\{\,w(D*A)\mid A\subseteq E\,\}.
\]
Arrange the elements of $\mathcal{W}$ in increasing order as $w_1<w_2<\dots<w_t$.
A single twist changes the width by at most $2$ (Table~\ref{tab:widths}), so we must have
\[
w_{i+1}-w_i\le 2\qquad\text{for all } i\in\{1,\dots,t-1\}.
\]

If all $w_i$ ($i=1,\dots,t$) have the same parity, then $\mathcal{W}$ contains no integer of the opposite parity.
Since $w_{i+1}-w_i\le 2$ and all numbers share the same parity, the sequence consists of consecutive odd (or even) integers. Hence the twist polynomial is an odd polynomial (if all $w_i$ are odd) or an even polynomial (if all $w_i$ are even).

Now suppose $\mathcal{W}$ contains both odd and even integers.
Without loss of generality, assume $w_1$ is odd. Let $w_k$ be the first even integer in the list.
Since $w_{k-1}$ is odd and $w_k$ is even, their difference is an odd number. Combined with $w_k-w_{k-1}\le 2$, we have
\[
w_k-w_{k-1}=1.
\]
Thus $w_{k-1}$ and $w_{k-1}+1$ are two consecutive twist widths.
By Theorem~\ref{theorem3}, for every integer $m$ with $w_{k-1}+2\le m\le w_M(D)$, we have $m\in \mathcal{W}$.
Consequently, all integers $m \ge w_{k-1}$ belong to $\mathcal{W}$. This implies that the even part and the odd part of the twist polynomial both have no gaps, so the polynomial is both even-interpolating and odd-interpolating.
\end{proof}

\begin{lemma} [\cite{AB4}]\label{lem 4}
Every ribbon-graphic delta-matroid is a binary delta-matroid.
\end{lemma}

\begin{lemma}[\cite{CH2019}]\label{lem 1}
Let $G$ be a ribbon graph and $A\subseteq E(G)$. Then $D(G^{A})=D(G)*A$ and $\varepsilon(G)=w(D(G))$.
\end{lemma}

\begin{corollary}
\label{cor:ribbon}
Let $G$ be a ribbon graph. If the set $\{\,\varepsilon(G^A)\mid A\subseteq E(G)\,\}$ contains two consecutive integers $k$ and $k+1$, then for every integer $m$ with $k+2\le m\le \varepsilon_M(G)$, we have $m\in \{\,\varepsilon(G^A)\mid A\subseteq E(G)\,\}$.
\end{corollary}

\begin{proof}
By Lemma~\ref{lem 4}, $D(G)$ is binary.
For any $A\subseteq E(G)$,  $\varepsilon(G^A)=w(D(G)*A)$ by Lemma~\ref{lem 1}.
Thus $$\{\varepsilon(G^A)\mid A\subseteq E(G)\} = \{w(D(G)*A)\mid A\subseteq E(G)\}.$$
The conclusion follows directly from Theorem~\ref{theorem3} applied to $D(G)$.
\end{proof}

Applying the same parity argument as in the proof of Theorem~\ref{theorem5} to the set $\{\varepsilon(G^A)\mid A\subseteq E(G)\}$, together with Corollary~\ref{cor:ribbon}, we obtain Theorem \ref{main theorem}.

\section*{Acknowledgements}
This work is supported by NSFC (No. 12471326).

\end{document}